\documentclass[]{article}
\usepackage{graphicx}
\usepackage{amsmath}
\usepackage{amssymb}
\usepackage{amsthm}
\usepackage{pxfonts}
\usepackage{enumerate}
\usepackage{color}
\usepackage{mathdots}
\usepackage{sectsty}
\usepackage[hidelinks]{hyperref}
\usepackage{tikz}
\usepackage{caption}
\usepackage{adjustbox}
\usepackage{fancyhdr}
\usepackage{verbatim}
\usepackage{url}

\sectionfont{\scshape\centering\fontsize{11}{14}\selectfont}
\subsectionfont{\scshape\fontsize{11}{14}\selectfont}

\newcommand\shorttitle{Primes in higher-order progressions on average}
\newcommand\authors{\small Nian Hong Zhou}

\hypersetup{
    colorlinks=true,       
    linkcolor=blue,          
    citecolor=cyan,        
}

\def\ssum{\mathop{\sum\!\sum}}

\def\qb{\mathbb{Q}}

\def\nb{\mathbb{N}}
\def\zb{\mathbb{Z}}

\renewcommand{\mod}{\mathop{\rm{mod}}}
\numberwithin{equation}{section}
\newtheorem{theorem}{Theorem}[section]
\newtheorem{lemma}[theorem]{Lemma}

\numberwithin{equation}{section}
\def\rint{\int\limits}

\title{\large \bf PRIMES IN HIGHER--ORDER PROGRESSIONS ON AVERAGE}
\author{\small NIAN HONG ZHOU}

\date{}

\begin{document}

\maketitle

\begin{abstract}
In this paper, we establish some theorems on the distribution of primes in higher-order progressions on average.
\end{abstract}


\section{Introduction}
The Bateman-Horn conjecture \cite{MR0148632} suggests that if $x^{\ell}+u\in\zb[x]$ be irreducible polynomial with $u$ be an even number and the degree $\ell\ge 1$, then
\begin{equation} \label{bhconj}
 \sum_{m \leq X} \Lambda(m)\Lambda(m^{\ell}+u) \sim \prod_p\left\{ \left( 1-\frac{1}{p}\right)^{-2}\left(1-\frac{n_{\ell}(p,u)}{p} \right)\right\} X,
 \end{equation}
where $\Lambda$ denotes the von Mangoldt function, $p$ stands for primes and $n_{\ell}(p,u)$ being the number of solutions of the congruence $x(x^{\ell}+u)\equiv 0 ~(\mod p)$.

If $\ell=1$, the asymptotic formula in \eqref{bhconj} is the twin prime conjecture. However, even the simple case seems beyond the current approach. In 1970, Lavrik \cite{MR0159801} proved that if $\ell=1$, then given any $A>0$, \eqref{bhconj} holds for all even integer $u\ge 1$ not exceeding $X$ with at most
$O \left( X (\log X)^{-A} \right)$ exceptions.

In \cite{MR2317757}, S. Baier and L. Zhao established certain theorems for the Bateman-Horn conjecture for quadratic polynomials on average. Their main result states the following. Given $A,B>0$, we have, for $x^2 (\log x)^{-A} \leq y \le x^2$,
\begin{equation*}
\sum_{\substack{k \le y\\ \mu^2(k)=1}} \left| \sum_{n \le x} \Lambda(n^2+k) - \mathfrak{S}(k) x \right|^2\ll {yx^2}{(\log x)^{-B}},
\end{equation*}
where
\[ \mathfrak{S}(k) = \prod_{p>2} \left( 1 - \frac{1}{p-1}\left(\frac{-k}{p} \right) \right) \]
with $\left( \frac{-k}{p} \right)$ being the Legendre symbol. In \cite{MR3054332}, F. Too and L. Zhao established similar results for the cubic cases.

In this paper, we shall study the asymptotic formula in \eqref{bhconj} on average.  Our main results are as follows.
\begin{theorem}\label{main}
Let integer $\ell\ge 2$. For any $A> 0$, there exists a $B=B_{\ell}(A)$ such that
\[\sum_{1 \le u \le y}\left| \sum_{m \le \sqrt[\ell]{X}}\Lambda(m^{\ell} + u)\Lambda(m) - {\frak S}_{\ell} (u)\sqrt[\ell]{X}\right|^2 \ll \frac{yX^{\frac{2}{\ell}}}{\log^{A}X}\]
holds for any $y\in (X^{1-\frac{1}{2\ell}}(\log X)^B,X]$, where
\[ {\frak S}_{\ell}(u) =\prod_{p|u}\frac{p-\varrho_{\ell}(p,u)}{p-1-\varrho_{\ell}(p,u)}\prod_{p}\left(1-\frac{\varrho_{\ell}(p,u)-1}{p-1}-\frac{\varrho_{\ell}(p,u)}{(p-1)^2}\right), \]
$p$ stands for primes and $\varrho_{\ell}(p,u)$ being the number of solutions of the congruence $x^{\ell}+u\equiv 0 ~(\mod p)$.
\end{theorem}
By similar arguments, we have the following theorem which improves the results in \cite{MR2317757} and \cite{MR3054332}.
\begin{theorem}\label{main1}
Let integer $\ell\ge 2$. For any $A> 0$ and $\varrho_{\ell}(p,u)$ as defined in the Theorem 1.1. Then there exists a $B^{\prime}=B^{\prime}_{\ell}(A)$ such that
\[\sum_{1 \le u \le y}\left| \sum_{m \le \sqrt[\ell]{X}}\Lambda(m^{\ell} + u) - {\frak S}_{\ell}^{\prime} (u)\sqrt[\ell]{X}\right|^2 \ll \frac{yX^{\frac{2}{\ell}}}{\log^{A}X}\]
holds for any $y\in (X^{1-\frac{1}{\ell}}(\log X)^{B^{\prime}},X]$ with
\[ {\frak S}_{\ell}^{\prime}(u) = \prod_{p}\left(1-\frac{\varrho_{\ell}(p,u)-1}{p-1}\right) \]
and the product being taken over all primes.
\end{theorem}
The primary technique used in the proof of Theorem \ref{main} is the circle method and the using of a variant of Weyl's inequality. The main difficulty in this application of the circle method is with the singular series. As for the asymptotic conjecture \eqref{bhconj}, the coefficient
${\frak S}_{\ell}(u)$ involves the using of Dedekind zeta functions associated to suitable algebraic number fields of the form $\qb[\sqrt[\ell]{u}]$. On the other hand, let $p, q$ denote primes and we observe that left of \eqref{bhconj} means that one can give an estimate for
\[\#\{q\in\nb: q^{\ell}+u=p, q\le X\}=\#\{q\in\nb: p-q^{\ell}=u, q\le X\}.\]
Which similar with the Hardy-Littlewood conjecture \cite{MR1555183}, say every sufficiently large number is either an $\ell$-$th$ power or a sum of
a prime number and an $\ell$-$th$ power, for $\ell = 2,3$. When the circle method be used, in fact there is no big difference between them. Therefore when $\ell\ge 2$, the singular series similar to the singular series of Zaccagnini \cite{MR1203294}, which first give a crude estimates for the kinds of singular series.  In \cite{MR1249438}, Kawada announced that he could obtain an asymptotic formula for the number of representations
of numbers as the sum of a prime and an $\ell$-$th$ power on average, and give a detailed proof in \cite{MR1650745} by use of the analytic properties of the Dedekind zeta function. Based on this result and under Generalized Riemann Hypothesis, Br\"{u}dern \cite{MR2477986} give an asymptotic formula for the number of representations of numbers as the sum of a prime and an $\ell$-$th$ power of a prime on average.

Furthermore, combined with the work of Perelli, Zaccagnini \cite{MR1328560} and Bauer \cite{MR1762244}, we can have a good treatment for the minor arcs. Hence we get the proof of our main theorem.
\paragraph{Notations.}
Notation is standard or otherwise introduced when appropriate. The symbols $\zb$ and $\qb$ denote the set of integers and rational numbers, respectively. $e(z)=e^{2\pi i z}$, the letter $p$ always denotes a prime. The symbol $\zb_q$
represents shorthand for the groups $\zb/q\zb$. Also, the shorthand for the multiplicative
group composed by reduced residue classes $(\zb/q\zb)^*$ is $\zb_q^*$. Denote by $\varphi$ and $\Lambda$ the Euler and von Mangoldt functions, respectively. For a large number $X$, denote $L=\log X$. For the sake of simplicity, we set
\[I_{\ell}(\alpha,z)=\sum_{z<m^{\ell}\le 2^{\ell}z}e\left(m^{\ell}\alpha\right),\;\; J_{\ell}(\alpha,z)=\sum_{z<m^{\ell}\le 2^{\ell} z}\Lambda(m)e\left(m^{\ell}\alpha\right),\]
\[I(\alpha,z)=\sum_{m\le 2^{\ell}z}e\left(-m\alpha\right)\;\mbox{and}\;J(\alpha,z)=\sum_{m\le 2^{\ell}z}\Lambda(m)e\left(-m\alpha\right).\] Further, we set
\[\lambda(q,u)=\frac{1}{q}\sum_{a\in\zb_q^*}\sum_{h\in\zb_q}e\left(\frac{a(h^{\ell}+u)}{q}\right)\;\mbox{and}\; A(q,u)=\frac{1}{\varphi(q)}\sum_{a\in\zb_q^*}\sum_{h\in\zb_q^*}e\left(\frac{a(h^{\ell}+u)}{q}\right).\]
It is easily seen that both $\lambda(q,u)$ and $A(q,u)$ are multiplicative function with respect to positive integer $q$. It is obvious that
\[\lambda(q,u)=\prod_{p|q}\left(\varrho_{\ell}(p,u)-1\right)\]
and
\[ A(q,u)=\frac{q}{\varphi(q)}\prod_{p|
{q}{(q,u)^{-1}}}\left(\varrho_{\ell}(p,u)-1+\frac{1}{p}\right)\prod_{p|(u,q)}\left(\varrho_{\ell}(p,u)-2+\frac{1}{p}\right)\]
when $q$ is square-free. Also, for any $z\ge 1$, we always set
\[{\frak S}_{\ell}^{\prime}(u,z)=\sum_{q\le z}\frac{\mu(q)}{\varphi(q)}\lambda(q,u),~~{\frak S}_{\ell}(u,z)=\sum_{q\le z}\frac{\mu(q)}{\varphi(q)}A(q,u),\]
\[{\frak P}_{\ell}(u,z)=\prod_{p|u, p\le z}\frac{p-\varrho_{\ell}(p,u)}{p-1-\varrho_{\ell}(p,u)}
\prod_{p\le z}\left(1-\frac{\varrho_{\ell}(p,u)-1}{p-1}-\frac{\varrho_{\ell}(p,u)}{(p-1)^2}\right)\]
and
\[{\frak P}_{\ell}^{\prime}(u,z)=\prod_{p\le z}\left(1-\frac{\varrho_{\ell}(p,u)-1}{p-1}\right).\]

\paragraph{Acknowledgements.}
The author would like to thank the anonymous referees for their very helpful
comments and suggestions. The author also thank Professor Zhi-Guo Liu for his consistent encouragement.
\section{Preliminary lemma}
We shall need the following well-known results in analytic number theory.
\begin{lemma}\label{t21}
Let $1\le z\le y$, $v\in \zb\setminus\{0\}$ and integer $\ell\ge 2$.
Then we have
\[\sum_{y<n\le 2y}\left|{\frak S}_{\ell}^{\prime}(nv,z)-{\frak S}_{\ell}^{\prime}(nv)\right|^2\ll_{v,\ell} yz^{-1/(2000\ell^2)}.\]
\end{lemma}
\begin{proof}
This is due to Theorem 1, Corollary 1 and Corollary 2 of \cite{MR1650745}.
\end{proof}
\begin{lemma}\label{t22}Let $|u|\ge 1$, $\ell\in\zb_{\ge 1}$ and $x^{\ell}+u$ is irreducible over $\qb[x]$. Then we have
\[\left|\sum_{p}\frac{\varrho_{\ell}(p,u)-1}{p}\right|\le O_{\ell}\left(1\right)+4\ell\log\log (2\left|u\right|). \]
\end{lemma}
\begin{proof}Let $D_u$ be the discriminant of $\qb[\sqrt[\ell]{u}]$. It is easily seen that $|D_u|\le \ell^{\ell}\left|u\right|^{\ell-1}.$  Hence by Landau prime ideal theorem (see \cite[Theorem 5.33]{MR2061214}) and partial summation we have
\begin{align*}
\left|\sum_{p}\frac{\varrho_{\ell}(p,u)-1}{p}\right|&\le (\ell-1)\sum_{p\le y}\frac{1}{p}+\left|\sum_{p>y}\frac{\varrho_{\ell}(p,u)-1}{p}\right|\\
&\le (\ell-1)\log\log y+O_{\ell}(1)+O\left(\sqrt{D_u}\exp\left(-c_{\ell}\sqrt{\log y}\right)\log y\right),
\end{align*}
where $c_{\ell}$ is an absolute constant depending only on $\ell$. Setting $y=\exp((\log (2|u|))^{4})$ we obtain that
\[\left|\sum_{p}\frac{\varrho_{\ell}(p,u)-1}{p}\right|\le 4\ell\log\log (2|u|)+ O_{\ell}\left(1+D_u^{\frac{1}{2}}|2u|^{-c_{\ell}\log |2u|}(\log |2u|)^4\right).\]
Thus if $|u|\ge \exp(\ell/c_{\ell})$, then we obtain that
\[\left|\sum_{p}\frac{\varrho_{\ell}(p,u)-1}{p}\right|\le 4\ell\log\log (2|u|)+ O_{\ell}\left(1\right).\]
Thus we get the proof of the lemma.
\end{proof}
\begin{lemma}\label{t23}Let $\alpha=a/q+\beta$, $\left|\beta\right|\le X^{-1}L^B$, $q\le L^B$ and $B\ge 1$. Also let $z\in(XL^{-B}, X]$. We have
\[ R_{\ell}(\alpha,z):=J_{\ell}(\alpha,z)-I_{\ell}(\beta, z)B_{\ell}(q,a)/\varphi(q)
\ll z^{1/\ell}L^{-10B^2},\]
\[R_{\ell}'(\alpha,z):=I_{\ell}(\alpha,z)-I_{\ell}(\beta, z)B_{\ell}'(q,a)/q
\ll z^{1/\ell}L^{-10B^2},\]
where
\[ B_{\ell}(q,a)=\sum_{h\in\mathbb{Z}_q^*}e\left({ah^{\ell}}/{q}\right)\;\mbox{and}\; B_{\ell}'(q,a)=\sum_{h\in\mathbb{Z}_q}e\left({ah^{\ell}}/{q}\right).\]
\end{lemma}
\begin{proof} It is easily seen that
\begin{align*}
J_{\ell}(\alpha,z)&=\sum_{h\in\mathbb{Z}_q^*}e\left(\frac{ah^{\ell}}{q}\right)\sum_{\substack{z^{1/\ell}<m\le 2z^{1/\ell}\\m\equiv h\bmod q}}\Lambda(m)e(m^{\ell}\beta)+O(\log q\log z)\\
&=\frac{1}{\varphi(q)}\sum_{h\in\mathbb{Z}_q^*}e\left(\frac{ah^{\ell}}{q}\right)\sum_{z^{1/\ell}<m\le 2z^{1/\ell}}e\left(m^{\ell}\beta\right)+R_{\ell}(\alpha,z),
\end{align*}
by partial summation and where
\[
 R_{\ell}(\alpha,z)\ll (\log z)^2+\varphi(q)(1+|\beta|z)\max_{h\in\zb_q^*}\max_{z^{1/\ell}< x\le 2z^{1/\ell}}\left|\sum_{\substack{m\le x\\m\equiv h\bmod q}}\Lambda(m)-x/\varphi(q)\right|.
 \]
 Then by \cite[Corollary 5.29]{MR2061214}, we get the estimate of $R_{\ell}(\alpha,z)$. The estimate of $R_{\ell}'(\alpha,z)$ is similar and we omit its detail.
\end{proof}
\begin{lemma}\label{t24}
Let $\alpha=a/q+\lambda$ with $(a,q)=1$ and $|\lambda|\le q^{-2}$. Then for each integer $\ell\ge 2$ and any $A>0$
there exists a $B_{\frak{m},\ell}(A)>0$ such that for $B\ge B_{\frak{m},\ell}(A)$ the estimate
\[
\int_{y}^{2y}{\rm d} t\left|\sum_{t<m^{\ell}\le t+H}\Lambda(m)e(m^{\ell}a/q)\right|^2\ll_{\ell,A} H^2y^{\frac{2}{\ell}-1}L^{-A-2}
\]
holds for $L^B<q\le HL^{-B}, y^{1-{1}/({2\ell})}L^B<H\le y$ and $y\ge \sqrt{X}$.
\end{lemma}
\begin{proof}
This is quoted from \cite[Lemma 3.3]{MR1762244}.
\end{proof}
\begin{lemma}\label{t25}
Let $a,q$ be positive integers with $(a,q)=1$. Then for each integer $\ell\ge 2$,
there exists a $B_{\frak{m},\ell}'(A)>0$ such that for $B\ge B_{\frak{m},\ell}'(A)$ the estimate
\[
W_{\ell}(y,H):=\int_{y}^{2y}{\rm d} t\left|\sum_{t<m^{\ell}\le t+H}e(m^{\ell}a/q)\right|^2\ll_{\ell,A} H^2y^{\frac{2}{\ell}-1}L^{-A-2}
\]
holds for $L^B<q\le HL^{-B}$, $y^{1-{1}/{\ell}}L^B<H\le y$ and $y\ge \sqrt{X}$.
\end{lemma}
\begin{proof} It is easily seen that
\begin{align*}
W_{\ell}(y,H)&=\ssum_{\substack{y<m_1^{\ell},m_2^{\ell}\le 2y+H\\ |m_1^{\ell}-m_2^{\ell}|\le H}}e\left((m_1^{\ell}-m_{2}^{\ell})a/q\right)\int\limits_{\max(m_1^{\ell}-H,m_2^{\ell}-H)}^{\min(m_1^{\ell},m_2^{\ell})}{\rm d}x\\
&=\ssum_{\substack{y<m_1^{\ell},m_2^{\ell}\le 2y+H\\ |m_1^{\ell}-m_2^{\ell}|\le H}}e\left((m_1^{\ell}-m_{2}^{\ell})a/q\right)\left(H-|m_1^{\ell}-m_2^{\ell}|\right)\\
&=\sum_{k}\sum_{n}\left(H-|P_{\ell}(n,k)|\right){\bf 1}_{\left(\substack{y<(n+k)^{\ell},n^{\ell}\le 2y+H\\ |(n+k)^{\ell}-n^{\ell}|\le H}\right)}e\left(P_{\ell}(n,k)a/q\right)\\
&\ll \sum_{k\le Hy^{1/\ell-1}}H\max_{y^{1/\ell}<x\le 2y^{1/\ell}}\left|\sum_{y^{1/\ell}\le n\le x}e\left(P_{\ell}(n,k)a/q\right)\right|\\
&\ll \sum_{k\le Hy^{1/\ell-1}}H\max_{y^{1/\ell}/2<x\le 2y^{1/\ell}}\left|\sum_{n\le x}e\left(P_{\ell}(n,k)a/q\right)\right|,
\end{align*}
where $P_{\ell}(n,k)=(n+k)^{\ell}-n^{\ell}$, then by \cite[Lemma]{MR1328560}, we have
\[W_{\ell}(y,H)\ll_{\ell,D} H^2y^{\frac{2}{\ell}-1}\left(\left(\frac{L^{D}}{q}\right)^{2^{2-\ell}}+\left(\frac{L^D}{y^{1/\ell}}\right)^{2^{2-\ell}}+\left(\frac{qL^D}{H}\right)^{2^{2-\ell}}+\left(\frac{L^{\ell^2}}{L^{D}}\right)^{2^{2-\ell}}\right) \]
holds for any $D>1$. By setting $D=B/2$ and $B_{{\frak m},\ell}'(A)=2^{\ell-1}(A+2)+2\ell^2$, we obtain the proof of the lemma.
\end{proof}
\section{The proof of the main results}
We first denote
\[S_{\ell}(y,X)= \sum_{1 \le u \le y}\left| \sum_{m^{\ell} \le X}\Lambda(m^{\ell} + u)\Lambda(m) - {\frak S}_{\ell} (u)\sum_{m^{\ell} \le X}1\right|^2\]
and
\[S_{\ell}^{\prime}(y,X)= \sum_{1 \le u \le y}\left| \sum_{m^{\ell} \le X}\Lambda(m^{\ell} + u) - {\frak S}_{\ell}^{\prime} (u)\sum_{m^{\ell} \le X}1\right|^2.\]
Then, by sum over dyadic intervals process one has
\begin{align}\label{eq31}
S_{\ell}(y,X)\ll& \sum_{1 \le u \le y}\left| \sum_{m^{\ell} \le XL^{-B}}\Lambda(m^{\ell} + u)\Lambda(m) - {\frak S}_{\ell} (u)\sum_{m^{\ell} \le XL^{-B}}1\right|^2 \nonumber\\
&+\sum_{1 \le u \le y}\left| \sum_{XL^{-B}<m^{\ell} \le X}\Lambda(m^{\ell} + u)\Lambda(m) - {\frak S}_{\ell} (u)\sum_{XL^{-B}<m^{\ell} \le X}1\right|^2 \nonumber\\
\ll& {yX^{2/\ell}}{L^{-B}}+BL\sup_{X/L^{B}\le z\le X/2^{\ell}}F_{\ell}(y,z),
\end{align}
where $B\ge 2$ and
\[F_{\ell}(y,z)=\sum_{u\le y}\left| \sum_{z<m^{\ell} \le 2^{\ell}z}\left(\Lambda(m^{\ell} + u)\Lambda(m) - {\frak S}_{\ell} (u)\right)\right|^2.\]
Similarly, we have
\begin{equation}\label{eq32}
S_{\ell}^{\prime}(y,X)\ll {yX^{2/\ell}}{L^{-B}}+BL\sup_{X/L^{B}\le z\le X/2^{\ell}}F_{\ell}^{\prime}(y,z)
\end{equation}
for any $B\ge 2$, where
\[F_{\ell}^{\prime}(y,z)=\sum_{u\le y}\left| \sum_{z<m^{\ell} \le 2^{\ell}z}\left(\Lambda(m^{\ell} + u) - {\frak S}_{\ell}^{\prime} (u)\right)\right|^2.\]
\newline
We define the major arcs as
\begin{equation}\label{eq33}
J_{q,a}=\left({a}/{q}-{L^{B^2}}/{X},{a}/{q}+{L^{B^2}}/{X}\right],
\end{equation}
where $1\le a\le q$. It is obvious that the interval $J_{q,a}$ are pairwise disjoint. Setting
\begin{equation}\label{eq34}
{\frak M}=\bigcup_{q\le L^{B^2}} \sideset{}{^{*}}\bigcup_{1\le a\le q} J_{q,a}\quad \mbox{and}\quad{\frak m}=\left(L^{B^2}/{X},1+L^{B^2}/{X}\right]\setminus{\frak M},
\end{equation}
where $*$ means that $(a,q)=1$. Application of the circle method gives
\[
\sum_{z<m^{\ell}\le 2^{\ell}z}\Lambda(m^{\ell}+u)\Lambda(m)=\left\{\int_{{\frak M}}+
\int_{{\frak m}}\right\}J(\alpha,z)J_{\ell}(\alpha,z)e\left(u\alpha\right){\rm d}\alpha.
\]
Therefore,
\begin{align}\label{eq35}
F_{\ell}(y,z)=&\sum_{u\le y}\left|\sum_{z<m^{\ell}\le 2^{\ell}z}\left(\Lambda(m^{\ell}+u)\Lambda(m)-{\frak S}_{\ell}(u)\right)\right|^2\nonumber\\
\ll& \sum_{u\le y}\left|\int_{{\frak M}}J(\alpha,z)J_{\ell}(\alpha,z)e\left(u\alpha\right){\rm d}\alpha-\sum_{z<m^{\ell}\le 2^{\ell}z}{\frak S}_{\ell}(u)\right|^2\nonumber\\
&+\sum_{u\le y}\left|\int_{{\frak m}}J(\alpha,z)J_{\ell}(\alpha,z)e\left(u\alpha\right){\rm d}\alpha\right|^2\\
=:&S_{{\frak M}}(y,z)+S_{{\frak m}}(y,z).
\end{align}
Similarly, we have
\begin{equation}\label{eq36}
F_{\ell}'(y,z)\ll S_{{\frak M}}'(y,z)+S_{{\frak m}}'(y,z),
\end{equation}
where
\[S_{{\frak M}}'(y,z)=\sum_{u\le y}\left|\int_{{\frak M}}J(\alpha,z)I_{\ell}(\alpha,z)e\left(u\alpha\right){\rm d}\alpha-\sum_{z<m^{\ell}\le 2^{\ell}z}{\frak S}_{\ell}'(u)\right|^2\]
and
\[S_{{\frak m}}'(y,z)=\sum_{u\le y}\left|\int_{{\frak m}}J(\alpha,z)I_{\ell}(\alpha,z)e\left(u\alpha\right){\rm d}\alpha\right|^2.\]

We shall prove the following lemmas, from which, \eqref{eq31}, \eqref{eq32}, \eqref{eq35} and \eqref{eq36} the results of our two theorems  follow.
\begin{lemma}\label{t31} For any $A>0$,there exists a $B_{\ell,1}(A)>0$ such that for $B\ge B_{\ell,1}(A)$, then
\[
S_{\frak m}(y,z)\ll yz^{{2}/{\ell}}L^{-A}
\]
holds for all $z^{1-{1}/{(2\ell)}}L^B\le y\le z$ with $XL^{-B}\le z\le X/2^{\ell}$.

For any $A>0$,there exists a $B_{\ell,1}'(A)>0$ such that for $B\ge B_{\ell,1}'(A)$, then
\[
S_{\frak m}'(y,z)\ll yz^{{2}/{\ell}}L^{-A}
\]
holds for all $z^{1-{1}/{\ell}}L^B\le y\le z$ with $X/L^{B}\le z\le X/2^{\ell}$..
\end{lemma}

\begin{lemma}\label{t32} Let $z\in[X/L^{B}, X/2^{\ell}]$ and $y\in (z^{\delta}, z]$ with $\delta\in(0,1)$ be fixed. We have
\[
S_{\frak M}(y,z)\ll_{\delta,\ell,A} yz^{\frac{2}{\ell}}L^{-A} \;\mbox{and}\; S_{\frak M}^{\prime}(y,z)\ll_{\delta,\ell,A} yz^{\frac{2}{\ell}}L^{-A}
\]
for any $A>0$, $B=\max(2000\ell^2(12\ell+A), 2^{\ell}(10\ell+A)+c(\ell))$ with $c(\ell)$ an absolute constant depending only on $\ell$.
\end{lemma}
\section{The minor arcs}
In this section, we shall prove Lemma \ref{t31}. Firstly, we have
\[S_{\frak m}(y,z)\ll \int_{{\frak m}}\left|J(\alpha,z)J_{\ell}(\alpha,z)\right|^2{\rm d}\alpha\ll z\sup_{\alpha\in {\frak m}}|J_{\ell}(\alpha,z)|^2 \]
by Bessel's inequity. Then the classical result
\[J_{\ell}(\alpha,z)\ll z^{{1}/{\ell}}(\log z)^{-B}\]
holds for all $\alpha\in {\frak m}$ and any $B\ge 0$. This implies that if $y\in(zL^{-B},z]$ then
\begin{equation}\label{eq41}
S_{\frak m}(y,z)\ll z^{1+{2}/{\ell}}L^{-2B}\ll yz^{{2}/{\ell}}L^{-B}.
\end{equation}
If $y\in (z^{1-{1}/{(2\ell)}}L^B, zL^{-B}]$, then
\begin{align*}
S_{\frak m}(y,z)&=\sum_{u\le y}\left|\int_{{\frak m}}J(\alpha,z)J_{\ell}(\alpha,z)e\left(u\alpha\right){\rm d}\alpha\right|^2\\
&=\int_{{\frak m}}{\rm d}\beta J(\beta,z)J_{\ell}(\beta,z)\int_{{\frak m}}J(-\alpha,z)J_{\ell}(-\alpha,z)\sum_{u\le y}e\left(u(\alpha-\beta)\right){\rm d}\alpha\\
&\ll \int_{{\frak m}}{\rm d}\beta \left|J(\beta,z)J_{\ell}(\beta,z)\right|\int_{{\frak m}}\left|J(\alpha,z)J_{\ell}(\alpha,z)\right|\min\left(y,\frac{1}{\parallel\alpha-\beta\parallel}\right){\rm d}\alpha.
\end{align*}
Splitting the unit interval in $H=\lfloor y\rfloor+1$ adjacent, disjoint intervals $H_i$ of length $H^{-1}$, we obtain that
\[
S_{\frak m}(y,z)\ll\ssum_{1\le i,j\le H}\frac{y}{1+\left|i-j\right|} \iint_{\substack{{\frak m}\cap H_i\\ {\frak m}\cap H_j}}{\rm d}\beta{\rm d}\alpha \left|J(\beta,z)J_{\ell}(\beta,z)J(\alpha,z)J_{\ell}(\alpha,z)\right|.
\]
By cauchy's inequity, we have
\begin{align*}
S_{\frak m}(y,z)&\ll y\sum_{1\le i\le H}\left( \int_{{\frak m}\cap H_i}{\rm d}\beta \left|J(\beta,z)J_{\ell}(\beta,z)\right|\right)^2\sum_{1\le j\le H}\frac{1}{1+\left|i-j\right|}\\
&\ll y\log y\sum_{1\le i\le H}\left( \int_{{\frak m}\cap H_i}{\rm d}\beta \left|J(\beta,z)\right|^2\int_{{\frak m}\cap H_i}{\rm d}\beta \left|J_{\ell}(\beta,z)\right|^2\right)\\
&\ll yL \int_{{\frak m}}\left|J(\alpha,z)\right|^2{\rm d}\alpha \max_{1\le i\le H}\int_{{\frak m}\cap H_i}{\rm d}\beta \left|J_{\ell}(\beta,z)\right|^2.
\end{align*}
For $\beta=a/q+\lambda\in {\frak m}\cap H_i (1\le i\le H)$, there
exist $q$, $a$ and $\lambda$ satisfying $\beta=a/q+\lambda$, $L^B\le q\le HL^{-B}$, $|\lambda|\le L^{-B}$ and $(q,a)=1$. Applying Gallagher's lemma (see \cite[Lemma 1]{MR0279049}) we have
\begin{align*}
\int_{{\frak m}\cap H_i}{\rm d}\beta \left|J_{\ell}(\beta,z)\right|^2&\ll \int_{\left|\lambda\right|\le \frac{1}{H}}{\rm d}\lambda \left|\sum_{z<m^{\ell}\le 2^{\ell}z}\Lambda(m)e\left(\frac{a}{q}m^{\ell}\right)e\left(m^{\ell}\lambda\right)\right|^2\\
&\ll \frac{1}{H^{2}}\int_{\mathbb{R}}{\rm d}x\left|\sum_{x\le m^{\ell}\le x+H/2}{\rm 1}_{z<m^{\ell}\le 2^{\ell}z}\Lambda(m)e\left(\frac{a}{q}m^{\ell}\right)\right|^2\\
&=\frac{1}{H^{2}}\int_{z-H/2}^{2^{\ell}z}{\rm d}x\left|\sum_{x\le m^{\ell}\le x+H/2}{\rm 1}_{z<m^{\ell}\le 2^{\ell}z}\Lambda(m)e\left(\frac{a}{q}m^{\ell}\right)\right|^2.
\end{align*}
Namely,
\begin{align*}
&\int_{{\frak m}\cap H_i}{\rm d}\beta \left|J_{\ell}(\beta,z)\right|^2\\
&\qquad\ll\frac{1}{H^{2}}\left(\int_{z}^{2^{\ell}z-H/2}{\rm d}x\left|\sum_{x\le m^{\ell}\le x+H/2}\Lambda(m)e\left(\frac{a}{q}m^{\ell}\right)\right|^2+H(Hz^{\frac{1}{\ell}-1})^2\right)\\
&\qquad\ll \frac{1}{H^{2}}\sum_{j=0}^{\ell-1}\int_{2^jz}^{2^{j+1}z}{\rm d}x\left|\sum_{x\le m^{\ell}\le x+H/2}\Lambda(m)e\left(\frac{a}{q}m^{\ell}\right)\right|^2+Hz^{\frac{2}{\ell}-2}.
\end{align*}
Then by Lemma \ref{t24} and notice that $y\le zL^{-B}$, one has
\[S_{\frak m}(y,z)\ll yL(z^{\frac{2}{\ell}-1}L^{-A-2}+z^{\frac{2}{\ell}-1}L^{-B})\int_{0}^{1}\left|J(\alpha,z)\right|^2{\rm d}\alpha\ll yz^{\frac{2}{\ell}}L^{-A}\]
holds for any $A>0$ if $B\ge \max(B_{{\frak m},\ell}(A),A)+2$. Combining \eqref{eq41} and above, we get
\[S_{\frak m}(y,z)\ll yz^{\frac{2}{\ell}}L^{-A}\]
holds for any $B\ge \max(B_{{\frak m},\ell}(A),A)+2$. Finally, using Lemma \ref{t25} in place of Lemma \ref{t24}, it is not difficult to obtain the proof of the estimate of $S_{\frak m}'(y,z)$.
\section{ The major arcs}
In this section we consider the estimates for $S_{{\frak M}}(y,z)$ and $S_{{\frak M}}'(y,z)$. For $S_{{\frak M}}(y,z)$, notice that the definition of $B_{\ell}(q,a)$ (see Lemma \ref{t23}), the fact
\[B_{1}(q,-a)=\sum_{h\in\zb_q^*}e\left(-\frac{ah}{q}\right)=\mu(q) \;\mbox{if}\;  (q,a)=1,\]
(\ref{eq33}) and (\ref{eq34}) implies that
\begin{align*}
\rint_{{\frak M}}&J(\alpha,z)J_{\ell}(\alpha,z)e\left(u\alpha\right){\rm d}\alpha-{\frak S}_{\ell}(u)\sum_{z<m^{\ell}\le 2^{\ell}z}1\\
&\quad=\sum_{q\le L^{B^2}}\sum_{a\in\mathbb{Z}_q^*}e\left(\frac{au}{q}\right)\rint_{-{L^{B^2}}/{X}}^{{L^{B^2}}/{X}}{\rm d}\lambda J(\alpha,z)\left(J_{\ell}(\alpha,z)-\frac{B_{\ell}(q,a)}{\varphi(q)}I_{\ell}(\lambda,z)\right)e\left(u\lambda\right)\\
&\quad\quad+\sum_{q\le L^{B^2}}\sum_{a\in\mathbb{Z}_q^*}e\left(\frac{au}{q}\right)\frac{B_{\ell}(q,a)}{\varphi(q)}
\rint_{-{L^{B^2}}/{X}}^{{L^{B^2}}/{X}}{\rm d}\lambda I_{\ell}(\lambda,z)\left(J(\alpha,z)-\frac{\mu(q)}{\varphi(q)}I(\lambda,z)\right)e\left(u\lambda\right)\\
&\quad\quad+\sum_{q\le L^{B^2}}\sum_{a\in\mathbb{Z}_q^*}e\left(\frac{au}{q}\right)\frac{\mu(q)B_{\ell}(q,a)}{\varphi(q)^2}\rint_{-{L^{B^2}}/{X}}^{{L^{B^2}}/{X}}{\rm d}\lambda I(\lambda,z)I_{\ell}(\lambda,z)e\left(u\lambda\right)-{\frak S}_{\ell}(u)\sum_{z<m^{\ell}\le 2^{\ell}z}1.
\end{align*}
Namely,
\begin{align*}
\rint_{{\frak M}}&J(\alpha,z)J_{\ell}(\alpha,z)e\left(u\alpha\right){\rm d}\alpha-{\frak S}_{\ell}(u)\sum_{z<m^{\ell}\le 2^{\ell}z}1\\
&\quad=\rint_{{\frak M}}{\rm d}\alpha J(\alpha,z)R_{\ell}(\alpha,z)e\left(u\alpha\right)+\rint_{{\frak M}}{\rm d}\alpha R(\alpha)J_{\ell}(\alpha,z)e\left(u\alpha\right)\\
&\quad\quad +\left({\frak S}_{\ell}(u, L^{B^2})\rint_{-{L^{B^2}}/{X}}^{{L^{B^2}}/{X}}{\rm d}\lambda I(\lambda,z)I_{\ell}(\lambda,z)e\left(u\lambda\right)
-{\frak S}_{\ell}(u)\sum_{z<m^{\ell}\le 2^{\ell}z}1\right)\\
&\quad=:\mathcal{I}_1(u,z)+\mathcal{I}_2(u,z)+\mathcal{I}_3(u,z),
\end{align*}
where
\[R(\alpha)=\sum_{m\le 2^{\ell}z}\Lambda(m)e\left(-m\alpha\right)-\frac{\mu(q)}{\varphi(q)}\sum_{m\le 2^{\ell}z}e\left(-m\lambda\right)\]
with $\alpha=a/q+\lambda$. Therefore, we have
\begin{align*}
S_{\frak M}(y,z)&=\sum_{u\le y}\left|\mathcal{I}_1(u,z)+\mathcal{I}_2(u,z)+\mathcal{I}_3(u,z)\right|^2\\
&\ll \sum_{u\le y}\left(\left|\mathcal{I}_1(u,z)\right|^2+\left|\mathcal{I}_2(u,z)\right|^2+\left|\mathcal{I}_3(u,z)\right|^2\right)
\end{align*}
by Cauchy's inequality. Similarly, we obtain that
\[
S_{\frak M}^{\prime}(y,z)\ll \sum_{u\le y}\left(\left|\mathcal{I}_1'(u,z)\right|^2+\left|\mathcal{I}_2'(u,z)\right|^2+\left|\mathcal{I}_3'(u,z)\right|^2\right),
\]
where
\[\mathcal{I}_1'(u,z)=\int_{{\frak M}}{\rm d}\alpha J(\alpha,z)R_{\ell}'(\alpha,z)e\left(u\alpha\right)\;,\; \mathcal{I}_2'(u,z)=\int_{{\frak M}}{\rm d}\alpha R(\alpha)I_{\ell}(\alpha,z)e\left(u\alpha\right)\]
and
\[\mathcal{I}_3'(u,z)={\frak S}_{\ell}'(u, L^{B^2})\int_{-\frac{L^{B^2}}{X}}^{\frac{L^{B^2}}{X}}{\rm d}\lambda I(\lambda,z)I_{\ell}(\lambda,z)e\left(u\lambda\right)
-{\frak S}_{\ell}'(u)\sum_{z<m^{\ell}\le 2^{\ell}z}1.\]

We have firstly
\begin{align*}
\sum_{u\le y}\left|\mathcal{I}_1(u,z)\right|^2&=\sum_{u\le y}\left|\sum_{q\le L^{B^2}}\sum_{a\in\mathbb{Z}_q^*}e\left(\frac{au}{q}\right)
\int_{-\frac{L^{B^2}}{X}}^{\frac{L^{B^2}}{X}}{\rm d}\lambda J(\lambda,z)R_{\ell}(\alpha,z)e\left(u\lambda\right)\right|^2\\
&\le \sum_{u\le y}\left|\sum_{q\le L^{B^2}}\varphi(q)X\sup_{\alpha\in{\frak M}}\left|R_{\ell}(\alpha,z)\right|X^{-1}L^{B^2}\right|^2\ll_B yz^{\frac{2}{\ell}}L^{-B^2}
\end{align*}
by Lemma \ref{t23}. Also, from lemma \ref{t23} we obtain that
\begin{align*}
\sum_{u\le y}\left|\mathcal{I}_2(u,z)\right|^2&=\sum_{u\le y}\left|\sum_{q\le L^{B^2}}\sum_{a\in\mathbb{Z}_q^*}e\left(\frac{au}{q}\right)
\int_{-\frac{L^{B^2}}{X}}^{\frac{L^{B^2}}{X}}{\rm d}\lambda J_{\ell}(\lambda,z)R(\alpha)e\left(u\lambda\right)\right|^2\\
&\ll \sum_{u\le y}\left|\sum_{q\le L^{B^2}}z^{\frac{1}{\ell}}\sum_{a\in\zb_q^*}\frac{\left|B(q,a)\right|}{\varphi(q)}\sup_{\alpha\in{\frak M}}\left|\sum_{j=0}^{\lfloor(\log (2X))/{\log 2}\rfloor+1}R_1(-\alpha,2^j)\right|\frac{L^{B^2}}{X}\right|^2\\
&\ll yz^{\frac{2}{\ell}}\left|\sum_{q\le L^{B^2}}L^{3B^2}\sup_{\alpha\in{\frak M}}\max_{z\le 2X}\left|R_1(-\alpha,z)\right|X^{-1}\right|^2  \ll yz^{\frac{2}{\ell}}L^{-B^2}.
\end{align*}
Note that $I(\lambda,z)\ll\left|\lambda\right|^{-1}$, Hua's inequity (see \cite[Theorem 4]{MR0194404})
\[\int_{0}^1\left|I_{\ell}(\lambda,z)\right|^{2^{\ell}}{\rm d}\lambda\ll_{\ell} z^{\frac{2^{\ell}-\ell}{\ell}} L^{c(\ell)},\]
where $c(\ell)$ is an absolute constant depending only on $\ell$. Then the using of H\"{o}lder's inequality gives
\begin{align*}
&\int_{\frac{L^{B^2}}{X}<\left|\lambda\right|\le \frac{1}{2}}{\rm d}\lambda \left|I(\lambda,z)I_{\ell}(\lambda,z)\right|\\
&\qquad\qquad\ll \left(\int_{\frac{L^{B^2}}{X}<\left|\lambda\right|\le \frac{1}{2}}{\rm d}\lambda \left|I(\lambda,z)\right|^{\frac{2^{\ell}}{2^{\ell}-1}}\right)^{\frac{2^{\ell}-1}{2^{\ell}}}\left(\int_{0}^1{\rm d}\lambda \left|I_{\ell}(\lambda,z)\right|^{2^{\ell}}\right)^{\frac{1}{2^{\ell}}}\\
&\qquad\qquad\ll_{\ell}\left({L^{B^2}}/{X}\right)^{(1-\frac{2^{\ell}}{2^{\ell}-1})\frac{2^{\ell}-1}{2^{\ell}}}\left(z^{\frac{2^{\ell}}{\ell}-1}L^{c(\ell)}\right)^{\frac{1}{2^{\ell}}}\ll z^{\frac{1}{\ell}}L^{-\frac{B-c(\ell)}{2^{\ell}}}.
\end{align*}
On the other hand,
\[
\int_{-\frac{1}{2}}^{\frac{1}{2}}{\rm d}\lambda I(\lambda,z)I_{\ell}(\lambda,z)e\left(u\lambda\right)=z^{\frac{1}{\ell}}+O(1).
\]
Setting $B\ge \max(2,c(\ell))$, we obtain that
\begin{align*}
\sum_{u\le y}\left|\mathcal{I}_3(u,z)\right|^2&=\sum_{u\le y}\left|{\frak S}_{\ell}(u,L^B)\int_{\left|\lambda\right|\le \frac{L^B}{X}}{\rm d}\alpha I(\alpha,z)I_{\ell}(\alpha,z)e(u\alpha)-{\frak S}_{\ell}(u)\sum_{z<m^{\ell}\le 2^{\ell}z}1\right|^2\\
&\ll z^{\frac{2}{\ell}}\sum_{u\le y}\left|{\frak S}_{\ell}(u,L^B)-{\frak S}_{\ell}(u)\right|^2+z^{\frac{2}{\ell}}L^{-\frac{B-c(\ell)}{2^{\ell}}}\sum_{u\le y}\left|{\frak S}_{\ell}(u)\right|^2.
\end{align*}
We can conclude from the above estimates that
\begin{equation}\label{eq51}
S_{\frak M}(y,z)\ll yz^{\frac{2}{\ell}}L^{-B}+z^{\frac{2}{\ell}}\sum_{u\le y}\left|{\frak S}_{\ell}(u,L^B)-{\frak S}_{\ell}(u)\right|^2+z^{\frac{2}{\ell}}L^{-\frac{B-c(\ell)}{2^{\ell}}}\sum_{u\le y}\left|{\frak S}_{\ell}(u)\right|^2
\end{equation}
for $B\ge \max(2,c(\ell))$.

We now prove the following crude estimates for ${\frak S}_{\ell}(u)$ and ${\frak S}_{\ell}^{\prime}(u)$.
\begin{lemma}\label{t51}For all integer $|u|\in (0, X]$, we have
\[{\frak S}_{\ell}(u)\ll_{\ell} L^{5\ell}\;\mbox{and}\; {\frak S}_{\ell}'(u)\ll_{\ell} L^{5\ell}.\]
\end{lemma}
\begin{proof}We just prove the estimate for ${\frak S}_{\ell}(u)$, the proof for ${\frak S}_{\ell}'(u)$ is similar. Note that $0\le \varrho_{\ell}(p,u)\le \ell$, we have
\begin{align*}
{\frak S}_{\ell}(u) &=\prod_{p|u}\frac{p-\varrho_{\ell}(p,u)}{p-1-\varrho_{\ell}(p,u)}\prod_{p}\left(1-\frac{\varrho_{\ell}(p,u)-1}{p-1}-\frac{\varrho_{\ell}(p,u)}{(p-1)^2}\right)\\
&\ll_{\ell} \prod_{p|u,p>2\ell}\left(1+\frac{1}{p-1-\ell}\right)\prod_{p>2\ell}\left(1-\frac{\varrho_{\ell}(p,u)-1}{p-1}\right)\\
&\ll_{\ell}\exp\left(\sum_{p|u,p>2\ell}\frac{1}{p}\right)\exp\left(\sum_{p>2\ell}\frac{1-\varrho_{\ell}(p,u)}{p}\right).
\end{align*}
Then by Lemma \ref{t22}, we trivially have
\[{\frak S}_{\ell}(u)\ll_{\ell}\exp\left(\log\log |u|+4\ell\log\log (2|u|)+O_{\ell}(1)\right)\ll_{\ell} \left(\log (2|u|)\right)^{5\ell}\ll L^{5\ell}\]
holds for all $|u|\in[1, X]\cap\zb$.  Which complete the proof of the lemma.
\end{proof}
Lemma \ref{t21}, Lemma \ref{t51}, \eqref{eq51} and the crude estimates ${\frak S}_{\ell}(u,x)\ll x$ for $x>0$ implies that
\begin{equation*}
S_{\frak M}(y,z)\ll yz^{\frac{2}{\ell}}L^{-B}+z^{\frac{2}{\ell}}L^{4B+10\ell}+z^{\frac{2}{\ell}}\sum_{L^{2B}<u\le y}\left|{\frak S}_{\ell}(u,L^B)-{\frak S}_{\ell}(u)\right|^2+yz^{\frac{2}{\ell}}L^{-\frac{B-c(\ell)}{2^{\ell}}+10\ell}.
\end{equation*}
Notice that $y\ge z^{\delta}\ge(X/L^B)^{\delta}\gg _{B,\delta} X^{\delta/2}$, we have
\begin{equation}\label{eq52}
S_{\frak M}(y,z)\ll_{\ell,\delta,B} yz^{\frac{2}{\ell}}L^{-\frac{B-c(\ell)}{2^{\ell}}+10\ell}+z^{\frac{2}{\ell}}\sum_{L^{2B}<u\le y}\left|{\frak S}_{\ell}(u,L^B)-{\frak S}_{\ell}(u)\right|^2.
\end{equation}
Similarly, we have
\begin{equation*}
S_{\frak M}^{\prime}(y,z)\ll_{\ell,\delta,B} yz^{\frac{2}{\ell}}L^{-\frac{B-c(\ell)}{2^{\ell}}+10\ell}+z^{\frac{2}{\ell}}\sum_{L^{2B}<u\le y}\left|{\frak S}_{\ell}^{\prime}(u,L^B)-{\frak S}_{\ell}^{\prime}(u)\right|^2.
\end{equation*}
From Lemma \ref{t21} we obtain the estimate for $S_{\frak M}^{\prime}(y,z)$ immediately, say
\begin{align*}
S_{\frak M}^{\prime}(y,z)\ll_{\ell,\delta,B} yz^{\frac{2}{\ell}}L^{-\frac{B-c(\ell)}{2^{\ell}}+10\ell}+z^{\frac{2}{\ell}}L y(L^B)^{-\frac{1}{2000\ell^2}}\ll_{\ell,\delta,A}  yz^{\frac{2}{\ell}}L^{-A}
\end{align*}
by setting $B=\max(2000\ell^2(1+A), 2^{\ell}(10\ell+A)+c(\ell))$. For get the estimate for $S_{\frak M}(y,z)$, we need the following lemma.
\begin{lemma}\label{t52}
Let positive real numbers $x$ and $y$ be sufficiently large. We have
\[\sum_{u\le y}\left|{\frak P}_{\ell}(u,x)-{\frak S}_{\ell}(u,x)\right|^2\ll_{\ell} yx^{-1}\log^{\ell^2} x+x^{4\log x}\]
and
\[\sum_{u\le y}\left|{\frak P}_{\ell}^{\prime}(u,x)-{\frak S}_{\ell}^{\prime}(u,x)\right|^2\ll_{\ell} yx^{-1}\log^{\ell^2} x+x^{4\log x}.\]
\end{lemma}
\begin{proof}
We denote by $P(x)=\prod_{p\le x}p$, $S(u,x)={\frak P}_{\ell}(u,x)-{\frak S}_{\ell}(u,x)$ and let $V>x$. Clearly,
\[
S(u,x)=\sum_{\substack{x<q\le V\\
q|P(x)}}\frac{\mu(q)}{\varphi(q)}A(q,u)+\sum_{\substack{q>V\\
q|P(x)}}\frac{\mu(q)}{\varphi(q)}A(q,u).
\]
Let $\lambda_x=\log^{-1}x>0$. We have the following estimate
\begin{align*}
\sum_{\substack{q>V\\
q|P(x)}}\frac{\mu(q)}{\varphi(q)}A(q,u)&\ll \sum_{\substack{q>V\\
q|P(x)}}\frac{\mu(q)^2}{\varphi(q)}\left|A(q,u)\right|\le \frac{1}{V^{\lambda_x}}\sum_{
q|P(x)}\frac{q^{\lambda_x}\mu(q)^2}{\varphi(q)}\left|A(q,u)\right|\\
&=\frac{1}{V^{\lambda_x}}\prod_{p\le x}\left(1+\frac{p^{\lambda_x}\left|A(p,u)\right|}{p-1}\right)\le \frac{1}{V^{\lambda_x}}\prod_{p\le x}\left(1+\frac{e\left|A(p,u)\right|}{p-1}\right).
\end{align*}
Setting $V=\exp(\log^2x)$ and notice that
\begin{align*}
A(p,u)&=\begin{cases}\frac{p}{p-1}(\varrho_{\ell}(p,u)-1)-1\quad &u\equiv 0(\bmod p)\\
\frac{p}{p-1}(\varrho_{\ell}(p,u)-1)+\frac{1}{p-1}\quad &u\not\equiv 0(\bmod p)
\end{cases}
\end{align*}
we obtain
\[|A(p,u)|\le \frac{\ell p}{p-1}.\]
Therefore we get
\begin{equation*}
\sum_{u\le y}\left|\sum_{q>V,
q|P(x)}\frac{\mu(q)}{\varphi(q)}A(q,u)\right|^2\ll y\frac{1}{x^2}\prod_{p\le x}\left(1+\frac{e\ell}{p}\right)^2\ll yx^{-2}\log^{2e\ell}x.
\end{equation*}
One the other hand,
\begin{align*}
\sum_{u\le y}\left|\sum_{\substack{x<q\le V\\
q|P(x)}}\frac{\mu(q)}{\varphi(q)}A(q,u)\right|^2&=\sum_{u\le y}\ssum_{\substack{x<q_1,q_2\le V\\ q_1,q_2|P(x)}}\frac{\mu(q_1)\mu(q_2)}{\varphi(q_1)\varphi(q_2)}A(q,u)\overline{A(q,u)}\\
&=\sum_{u\le y}\sum_{\substack{x<q\le V\\ q|P(x)}}\frac{\mu(q)^2}{\varphi(q)^2}\left|A(q,u)\right|^2+T_R(x)
\end{align*}
with
\begin{align*}
T_R(x)&=\ssum_{\substack{x<q_1\neq q_2\le V\\ q_1,q_2|P(x)}}\frac{\mu(q_1)\mu(q_2)}{\varphi(q_1)^2\varphi(q_2)^2}\ssum_{\substack{a_i\in\zb_{q_i}^*\\i=1,2}}\ssum_{\substack{h_j\in\zb_{q_j}^*\\j=1,2}}e\left(\frac{a_1h_1^{\ell}}{q_1}-\frac{a_2h_2^{\ell}}{q_2}\right)\sum_{u\le y}e\left(\frac{a_1q_2-a_2q_1}{q_1q_2}u\right)\\
&\quad\le\ssum_{\substack{x<q_1\neq q_2\le V\\ q_1,q_2|P(x)}}\frac{\mu(q_1)^2\mu(q_2)^2}{\varphi(q_1)^2\varphi(q_2)^2}\ssum_{\substack{a_i\in\zb_{q_i}^*\\i=1,2}}\ssum_{\substack{h_j\in\zb_{q_j}^*\\j=1,2}}\left(\sum_{\lfloor({q_1q_2})^{-1}{y}\rfloor q_1q_2<u\le y}1\right)\\
&\quad\le\ssum_{\substack{x<q_1\neq q_2\le V\\ q_1,q_2|P(x)}}q_1q_2\le \left(\sum_{x<q\le V,q|P(x)}q\right)^2\ll V^4\le x^{4\log x},
\end{align*}
where the obvious fact $q_1q_2\nmid (a_1q_2-a_2q_1)$ has been used. Moreover,
\begin{align*}
\sum_{\substack{x<q\le V\\ q|P(x)}}\frac{\mu(q)^2}{\varphi(q)^2}\left|A(q,u)\right|^2
&\ll x^{-1}\sum_{\substack{x<q\le V\\ q|P(x)}}\frac{\mu(q)^2q}{\varphi(q)^2}\left|A(q,u)\right|^2\ll x^{-1}\sum_{ q|P(x)}\frac{\mu(q)^2q}{\varphi(q)^2}\left|A(q,u)\right|^2\\
&\ll x^{-1}\prod_{p\le x}\left(1+\frac{p|A(p,u)|^2}{(p-1)^2}\right)\ll x^{-1}\log^{\ell^2} x.
\end{align*}
Hence we obtain that
\[\sum_{u\le y}\left|S(u,x)\right|^2\ll_{\ell} yx^{-1}\log^{\ell^2} x+x^{4\log x}+yx^{-2}\log^{2e\ell}x\ll_{\ell} yx^{-1}\log^{\ell^2} x+x^{4\log x}.\]
Similarly,
\[\sum_{u\le y}\left|{\frak P}_{\ell}^{\prime}(u,x)-{\frak S}_{\ell}^{\prime}(u,x)\right|^2\ll_{\ell} yx^{-1}\log^{\ell^2} x+x^{4\log x}.\]
Which completes the proof of the lemma.
\end{proof}
Under Lemma \ref{t52}, we have the following estimate for ${\frak S}_{\ell}(u,L^B)$.
\begin{lemma}\label{t53}Let $y\le X$ be sufficiently large.  We have
\[
\sum_{L^{2B}<u\le y}\left|{\frak S}_{\ell}(u,L^B)-{\frak S}_{\ell}(u)\right|^2\ll_{\ell} yL^{-B/(2000\ell^2)+10\ell+3}.
\]
\end{lemma}
\begin{proof} First of all, by Lemma \ref{t52} it is clear that
\[
\sum_{L^{2B}<u\le y}\left|{\frak S}_{\ell}(u,L^B)-{\frak S}_{\ell}(u)\right|^2\ll_{\delta,\ell, B} yL^{-B+1}+\sum_{L^{2B}<u\le y}\left|{\frak S}_{\ell}(u)-{\frak P}_{\ell}(u, L^B)\right|^2.
\]
Note that
\[
{\frak P}_{\ell}(u,x)={\frak P}_{\ell}^{\prime}(u,x)f_{\ell}(u,x),
\]
where
\[f_{\ell}(u,x)=\prod_{p|u,p\le x}\left(1-\frac{1}{p-\varrho_{\ell}(p,u)}\right)^{-1}
\prod_{p\le x}\left(1-\frac{\varrho_{\ell}(p,u)}{(p-1)(p-\varrho_{\ell}(p,u))}\right).\]
Let $x\rightarrow\infty$, then
\[
{\frak S}_{\ell}(u)={\frak S}_{\ell}^{\prime}(u)f_{\ell}(u),
\]
where $f_{\ell}(u)=\lim_{x\rightarrow\infty}f_{\ell}(u,x)$. It is easily seen that
\[f_{\ell}(u,x)\ll_{\ell}\log(|u|+2)\]
for all $x>0$ and integer $u\neq 0$. Hence by Lemma \ref{t51}, we obtain
\begin{align*}
\sum_{L^{2B}<u\le y}\left|{\frak P}_{\ell}(u,L^B)-{\frak S}_{\ell}(u)\right|^2&=\sum_{L^{2B}<u\le y}\left|{\frak P}_{\ell}^{\prime}(u,L^B)f_{\ell}(u,L^B)-{\frak S}_{\ell}^{\prime}(u)f_{\ell}(u)\right|^2\\
&\ll L^2\sum_{L^{2B}<u\le y}\left|{\frak P}_{\ell}^{\prime}(u,L^B)-{\frak S}_{\ell}^{\prime}(u)\right|^2+L^{10\ell}R_{f}.
\end{align*}
where it is not difficult prove that
\[R_f=\sum_{u\le y}\left|f_{\ell}(u,L^B)-f_{\ell}(u)\right|^2\ll_{\ell} yL^{-B+2}.\]
By Lemma \ref{t52}, we obtain that
\begin{align*}
\sum_{L^{2B}<u\le y}\left|{\frak P}_{\ell}^{\prime}(u,L^B)-{\frak S}_{\ell}^{\prime}(u)\right|^2\ll_{\delta,\ell, B} yL^{-B+1}+\sum_{L^{2B}<u\le y}\left|{\frak S}_{\ell}^{\prime}(u)-{\frak S}_{\ell}^{\prime}(u, L^B)\right|^2.
\end{align*}
Then, the following is obvious by Lemma \ref{t21}.
\end{proof}
Finally, using Lemma \ref{t53} and setting $B=\max(2000\ell^2(12\ell+A), 2^{\ell}(10\ell+A)+c(\ell))$ in \eqref{eq52} completes the proof of Lemma \ref{t32}.


\bigskip
\noindent
{\sc Department of Mathematics, East China Normal University\\
500 Dongchuan Road, Shanghai 200241, PR China}\newline
\href{mailto:nianhongzhou@outlook.com}{\small nianhongzhou@outlook.com}
\end{document}